\documentclass{article}
\usepackage[utf8]{inputenc}
\usepackage{authblk}
\usepackage{setspace}
\usepackage[margin=1.25in]{geometry}
\usepackage{graphicx}
\graphicspath{ {./figures/} }
\usepackage{subcaption}
\usepackage{amsmath}
\usepackage[english]{babel}
\usepackage{amsthm}
\newtheorem{theorem}{Theorem}[section]
\newtheorem{lemma}[theorem]{Lemma}

\title{Asymptotics of the Minimal Feedback Arc Set in Erd\H{o}s-R\'{e}nyi 
Graphs}
\author[1*]{Harvey Diamond}
\author[2$\dag$]{Mark Kon}
\author[3]{Louise Raphael}

\affil[1]{Department of Mathematics, West Virginia University, Morgantown, WV 26506}
\affil[2]{Department of Mathematics and Statistics, Boston University}
\affil[3] {Department of Mathematics, Howard University}
\affil[*]{Address correspondence to: diamond@math.wvu.edu}
\affil[$\dag$] {Research partially supported by National Science Foundation grant DMS-23-19011.}

\date{}

\begin{document}

\maketitle

\begin{abstract}
Given a directed graph, the Minimal Feedback Arc Set (FAS) problem asks for a minimal set of arcs which, when removed, results in an acyclic graph. Equivalently, the FAS problem asks to find an ordering of the vertices that minimizes the number of feedback arcs. The FAS problem is considered an algorithmic problem of central importance in discrete mathematics. Our purpose in this paper is to consider the problem in the context of Erd\H{o}s-R\'{e}nyi random directed graphs, denoted $D(n,p)$, in which each possible directed arc is included with a fixed probability $p>0$. Our interest is the typical ratio of the number of feedforward arcs to the number of feedback arcs that are removed in the FAS problem.  We show that as the number $n$ of vertices goes to infinity the probability that this ratio is greater than $1+\epsilon$ for any fixed $\epsilon > 0$ approaches zero. Similarly, letting $p$ go to zero as $n\rightarrow \infty$ this result remains true if $p>C\log{n}/n$ where $C$ depends on $\epsilon$.  
\end{abstract}


\section{Introduction}
The Minimal Feedback Arc Set (FAS) problem is longstanding in the field of graph theory and its applications. We note in this regard the recent survey by Kudelić [Ku] of prior results over the past 50 years.  
In a directed graph (digraph) $\textbf{D} =(V,E)$ with $V$ the set of vertices and $E$ the set of directed arcs, the FAS problem is to find a minimal set of arcs $E'$ such that removing these arcs produces a digraph $\textbf{D}'=(V,E \setminus E')$ that is acyclic. As noted in the classic text of Harary [Ha], a digraph is acyclic if and only if it is possible to order the vertices of $\textbf{D}$ so that the adjacency matrix $A(\textbf{D})$ of $\textbf{D}$ is upper triangular. The matrix $A(\textbf{D})$ is a $0/1$ matrix where a $1$ in the $(i,j)$ position corresponds to an arc  from vertex $i$ to vertex $j$. Hence a $1$ where $j>i$ above the diagonal indicates an arc from a vertex $i$ to a higher-numbered vertex $j$. A $1$ below the diagonal corresponds to an arc from a higher-numbered vertex to a lower numbered vertex. The latter are called feedback arcs. Our ongoing conceptual framework is that of the adjacency matrix of a digraph. Given $A(\textbf{D})$ for the graph $\textbf{D}=(V,E)$ the FAS problem is to find the minimum number of $1's$ to eliminate from $A(\textbf{D})$ in order that the resulting adjacency matrix can be transformed into an upper triangular matrix by renumbering the vertices. In turn, this is equivalent to finding a renumbering of the vertices that results in a minimal number of $1's$ below the diagonal. These $1's$ identify  the feedback arcs to be eliminated. 

Our graphs will be random directed graphs on $n$ vertices of the Erd\H{o}s-R\'{e}nyi type. The original Erd\H{o}s-R\'{e}nyi formulation involved undirected graphs, denoted as $G(n,p)$. Here we use the obvious extension to directed graphs, denoted here $D(n,p)$, as for instance in [Hef], in which each possible directed arc is included with a fixed probability $p$. Directed graphs themselves will be denoted $\textbf{D}$ (using boldface). Since the location of an arc in the adjacency matrix relative to the diagonal is important, we will introduce $0/1$ Bernoulli random variables $\{\textbf{x}_{ij},j>i;i,j=1,2,...,n\}$ and $\{\textbf{y}_{ij},i>j;i,j,=1,2,...,n\}$ whose values will populate our adjacency matrix above and below the diagonal respectively. These random variables are independent, identically distributed (iid) and equal to $1$ with probability $p$ and equal to $0$ with probability $q=1-p$.

We define $\textbf{X}=\sum{ \textbf{x}_{ij}},j>i$ and $\textbf{Y}=\sum{ \textbf{y}_{ij}},i>j$, 
which are binomial random variables with distribution $B(K,p)$ where $K=(n^2-n)/2$. In short, $\textbf{X}$ is the number of 1's that appear above the diagonal of the adjacency matrix, and $\textbf{Y}$ is the number that appear below. If we have simply an adjacency matrix $A(\textbf{D})$ then we write the entries without boldface so that $\{x_{ij},j>i\},\{y_{ij},i>j\}$ are the scalar $0/1$ entries and X,Y are the scalar totals above and below the diagonal respectively. 

Now we can consider the FAS problem in terms of the adjacency matrix in the following way: Permute the numbering of the vertices in the graph $\textbf{D}=(V,E)$ so that there are as few $1`s$ as possible below the diagonal of the adjacency matrix $A(\textbf{D})$ of $\textbf{D}$. The Feedback Arc Set of the graph is then all of the arcs corresponding to $1's$ below the diagonal. As is standard practice, $1's$ above the diagonal, which represent arcs from vertex $i$ to vertex $j$ with $j>i$, are referred to as ``feedforward arcs" and $1's$ below the diagonal, running backwards in the vertex ordering, are referred to as ``feedback arcs". Thus the FAS problem is to find a renumbering of the vertices that minimizes the number of feedback arcs. 

An important and general question at this point is: What is the typical fraction of arcs contained in the Feedback Arc Set of a given digraph? As above, this will always be asked after an optimal renumbering of the vertices which minimizes the number of such feedback arcs.  In the context of Erd\H{o}s-R\'{e}nyi graphs $D(n,p)$ this is the probabilistic question that we explore in the subsequent sections. In particular we explore conditions under which the minimal feedback arc set  will asymptotically (as $n \rightarrow \infty$) approach one half of all arcs with probability approaching 1. In such cases there is little to gain from solving the FAS problem,  since for any adjacency matrix corresponding to $\textbf{D}=(V,E)$, one could simply remove all the $1's$ lying below the diagonal, or, if their number were smaller, remove the $1's$ above the diagonal, either of which would lead to a trivial solution asymptotically equivalent to fully solving the FAS problem. 

\section{Graphs with a fixed probability $p$ of an arc as $n \rightarrow \infty$}
We introduce next our formulation of the problem for directed Erd\H{o}s-R\'{e}nyi graphs $D(n,p)$.  In this section $p$ is any fixed value satisfying $0<p<1$. In more refined estimates that follow we will assume that $0 < p \le 1/2$. Certainly the cases of interest are those with a relatively small number of arcs relative to the number possible, and so a small value of $p$. 

\par \noindent For $r=1+\epsilon \ge 1$ we define the set $E_r$ of $n \times n$ adjacency matrices $A$ by
\begin{equation}
    E_r=\{A:\frac{X}{Y} \ge r\}.
\end{equation} 
Note: $\frac{X}{Y} \ge r$ is defined to be true if $Y=0$.

\noindent For any given graph with adjacency matrix $A$ we define $A^*$ as the adjacency matrix that results from a solution of the FAS problem, with corresponding numbers $X^*$ and $Y^*$ of ones above and below the diagonal respectively. We then define 
\begin{equation}
E^*_r=\{A^*:\frac{X^*}{Y^*} \ge r\}.
\end{equation}
 In terms of digraphs $E_r$ is the set of digraphs for which the ratio of feedforward to feedback arcs is at least $r$ and $E^*_r$ is the set of optimized digraphs for which the ratio of feedforward to feedback arcs is at least $r$ \textbf{in an FAS solution}. It is clear that $E^*_r\subset E_r$ as the smaller set consists of optimized feedback sets. 

This brings us to our first result as applied to the probabilistic setting of Erd\H{o}s-R\'{e}nyi graphs:
\begin{lemma}
In an Erd\H{o}s-R\'{e}nyi graph on n vertices with probability $p$ of a given directed arc, the probability of obtaining a graph for which 
$\displaystyle{\frac{\textbf{X}^*}{\textbf{Y}^*} \ge r}$ in an FAS solution is no larger than $n!$ times the probability of obtaining a graph with an adjacency matrix satisfying $\displaystyle{\frac{\textbf{X}}{\textbf{Y}} \ge r}$ or: $$Pr\left(\frac{\textbf X^*}{\textbf Y^*} \ge r\right) \le n!Pr\left(\frac{\textbf X}{\textbf Y} \ge r\right). $$
\end{lemma}
\begin{proof}
The adjacency matrices A for which $\displaystyle{\frac{X^*}{Y^*} \ge r}$ consists precisely of all the permutations of the adjacency matrices in the set $E^*_r$. There are $n!$ such permutations. The set of adjacency matrices $\displaystyle{\{A:\frac{\textbf{X}^*}{\textbf{Y}^*} \ge r\}}$ is contained in the set all permutations of the matrices in the set $\displaystyle{\{A:\frac{\textbf{X}}{\textbf{Y}} \ge r\}}$, since $E^*_r\subset E_r$. Hence 
$$Pr\left(\frac{\textbf{X}^*}{\textbf{Y}^*} \ge r\right) \le n!Pr\left(\frac{\textbf{X}}{\textbf{Y}} \ge r\right).$$
\end{proof}

One can at this point, give a simple argument based on tails of the binomial distribution, to show that $$Pr\left(\frac{\textbf{X}}{\textbf{Y}} \ge 1+\epsilon\right)=O(\exp(-C_{p,\epsilon} K))$$
for some constant $C_{p,\epsilon}>0$ that depends on the fixed constants $p>0,\epsilon>0$ and is independent of $n$. Since $\displaystyle{K=\frac{n^2-n}{2}=O(n^2)}$ we would have in that case for any fixed $\epsilon,p$  
$$Pr\left(\frac{\textbf{X}^*}{\textbf{Y}^*} \ge 1+\epsilon\right) \le n!Pr\left(\frac{\textbf{X}}{\textbf{Y}} \ge 1+\epsilon\right) \le n!O(\exp(-C_{p,\epsilon} K)) \rightarrow 0 \text { as } n \rightarrow \infty$$ since $n!=O(\exp(n\log n)) \text{ and } K=O(n^2)$. The argument to follow is simply the observation that $\textbf{X}$, as the sum of $K$ Bernoulli random variables, has very small probability of being larger than  $E[\textbf{X}](1+\delta)$ and similarly $\textbf{Y}$ is smaller than $E[\textbf{Y}](1-\delta)$ with very small probability. This is presented in the proof of Theorem 2.3 below. 

Tails of the binomial distribution can be conveniently estimated by applying the Hoeffding inequality [Ho].
\begin{theorem}
    (Hoeffding, 1963) Let $\textbf{X}_1,..,\textbf{X}_n$ be independent random variables satisfying $a_i \le \textbf{X}_i \le b_i$ almost surely. Set $\textbf{S}_n=\textbf{X}_1+...+\textbf{X}_n$. Then
    $$Pr(\textbf{S}_n-E[\textbf{S}_n] \ge t) \le \exp \left(-\frac{2t^2}{\sum_{i=1}^{n} (b_i - a_i)^2}\right).$$
\end{theorem}
We can now obtain our first theorem: 
\begin{theorem}
    Let $0<p<1,\epsilon>0$ be fixed and define $r=1+\epsilon$. In a directed Erd\H{o}s-R\'{e}nyi random graph with $n$ vertices, let $\textbf{X}^*$ be the number of feedforward arcs and let $\textbf{Y}^*$ be the number of feedback arcs identified in the FAS solution. Then for any such $p,\epsilon$ we have $$Pr\left( \frac{\textbf{X}^*}{\textbf{Y}^*} \ge r \right) \rightarrow 0  \text{ as } n \rightarrow \infty.$$
\end{theorem}
\begin{proof}
Applying the Hoeffding inequality we can estimate for some small $\delta>0$:
$$Pr(\textbf{X} \ge Kp(1+\delta)) \le \exp \left(-\frac{(Kp \delta)^2}{K}\right)=\exp(-p^2 \delta^2 K)$$
and then applying the same inequality to $-\textbf{Y}$ we can similarly obtain:
$$Pr(\textbf{Y} \le Kp(1-\delta)) \le \exp(-\frac{(Kp \delta)^2}{K}=\exp(-p^2 \delta^2 K)$$
from which it follows
$$Pr(\textbf{X} < Kp(1+\delta)) >1-\exp(-p^2 \delta^2 K)$$
$$Pr(\textbf{Y} > Kp(1-\delta)) >1-\exp(-p^2 \delta^2 K)$$
$$Pr\left(\frac{\textbf{X}}{\textbf{Y}} <\frac{1+\delta}{1-\delta}\right) \ge Pr\left((\textbf{X} < Kp(1+\delta))\cap (\textbf{Y} > Kp(1-\delta))\right)> 1-2\exp(-p^2 \delta^2 K).$$
And so, taking complements, and choosing $\displaystyle{\delta=\frac{\epsilon}{2+\epsilon}}$ so that $\displaystyle{\frac{1+\delta}{1-\delta}=1+\epsilon}$ we have  
\begin{equation}
Pr\left(\frac{\textbf{X}}{\textbf{Y}}  \ge 1+\epsilon\right)  \le  2\exp\left(-p^2 \left(\frac{\epsilon}{2+\epsilon}\right)^2 K\right).
\end{equation}

From Lemma 2.1 we have $$Pr\left( \frac{\textbf{X}^*}{\textbf{Y}^*} \ge r \right) \le n!Pr\left( \frac{\textbf{X}}{\textbf{Y}} \ge r \right)\le n!2\exp \left(-p^2 \left(\frac{\epsilon}{2+\epsilon}\right)^2 K\right) \rightarrow 0 \text{ as } n \rightarrow \infty $$
using $\displaystyle{K=\frac{n^2-n}{2}}$. 
\end{proof}

\section{Large deviation results with $p\rightarrow 0 \text{ as } n \rightarrow \infty$}

If we reformulate the scalar condition $\displaystyle{\frac{X}{Y} \ge r}$ as $X-rY \ge 0$ and apply it to to our random variables, then $E[\textbf{X}-r\textbf{Y}]=-Kp\epsilon$ and we can re-express $\textbf{X}-r\textbf{Y} \ge 0$ as $(\textbf{X}-r\textbf{Y})-E[\textbf{X}-r\textbf{Y}] \ge -E[\textbf{X}-r\textbf{Y}]$. Such an event is a special case of a large deviation, where for a general random variable $\textbf{U}$ we attempt to bound $Pr(\textbf{U} \ge E[\textbf{U}](1+t))$ in terms of $t$. 

If we are to allow $p$ to vary with $n$ we should try to obtain sharper inequalities on $Pr(\textbf{X}-r\textbf{Y} \ge 0)$. In this regard there are a number of other large deviation inequalities one might apply. Here we will use a version of Bennett's inequality [Be] in a recent paper by Boucheron et. al. [Bou]:

\begin{theorem}
    (Bennett, 1962) Let $\textbf{X}_1,\textbf{X}_2,..,\textbf{X}_n$ be mean-zero independent random variables satisfying $\textbf{X}_i \leq 1$. Define $\sigma_i^2=Var(\textbf{X}_i)$ and $\sigma^2=\sum_{i=1}^{n}\sigma_i^2$. Let $\textbf{S}=\sum_{i=1}^{n} \textbf{X}_i$. Then  for any $t>0$ the following holds 
    $$Pr\left(\textbf{S} \ge t\right)\leq \exp\left(-\sigma^2h\left(\frac{t}{\sigma^2}\right)\right)$$
    where $h(u)=(1+u)\log(1+u)-u $. 
\end{theorem}

If we apply this result to the problem at hand, we can obtain the inequalities in the following theorem: 
\begin{theorem}  If $p\le 1/2$ and $rp \le 1$ we have:
\begin{equation}
    Pr\left(\textbf{X}-r\textbf{Y} \ge 0\right)\leq \exp\left(-(1+r^2)Kpqh\left(\frac{\epsilon }{(1+r^2)q}\right)\right).
\end{equation}
\begin{equation}
Pr\left(\textbf{X}-r\textbf{Y} \ge 0\right)\leq \exp\left(-\frac{Kp\epsilon^2 }{2(1+r^2)q}(1+O(\epsilon))\right).
\end{equation}
\begin{proof}

In applying Bennett's theorem to the case of $\textbf{X}-r\textbf{Y}$ we have the sum of the following mean-zero random variables, with associated variances: 
$$\{\textbf{x}_{ij}-p\},\{-r(\textbf{y}_{ij}-p)\},Var(\textbf{x}_{ij}-p)=pq,Var(-r(\textbf{y}_{ij}-p))=r^2pq,\sigma^2=K(1+r^2)pq.$$
We have $\textbf{x}_{ij}-p \le 1-p \le 1, -r(\textbf{y}_{ij}-p) \le rp \le 1$ provided $rp\le 1$ which we assume for the purposes here. Then, we define $$\displaystyle{\textbf{S}=\sum_{i<j} (\textbf{x}_{ij}-p)-r\sum_{i>j} (\textbf{y}_{ij}-p)=\textbf{X}-r\textbf{Y}+Kp\epsilon}.$$ 
Applying Bennet's inequality to $Pr(\textbf{S} \ge Kp\epsilon)$ we have 
$$Pr(\textbf{S} \ge Kp\epsilon)=Pr(\textbf{X}-r\textbf{Y} \ge 0) \le \exp\left(-(1+r^2)Kpqh\left(\frac{\epsilon }{(1+r^2)q}\right)\right) $$
which is equation (4). Now when $u$ is small we have $h(u)=\frac{1}{2}u^2+O(u^3)$ so when $\epsilon$ is small and $p\leq 1/2$, with $rp \le 1$, we have 

$$Pr\{\textbf{X}-r\textbf{Y} \ge 0\}=Pr\{\textbf{S} \ge  Kp\epsilon\}\leq \exp\left(-\frac{Kp\epsilon^2 }{2(1+r^2)q}(1+O(\epsilon))\right)$$

\par\noindent which completes the proof. 
\end{proof}
\end{theorem}
\par\noindent In comparing (5) with (3) when $p,\epsilon$ are small, we have $\displaystyle{p^2 (\frac{\epsilon}{2+\epsilon})^2\approx p^2\epsilon^2/4} $ versus $\displaystyle{\frac{p\epsilon^2 }{2(1+r^2)q}\approx p\epsilon^2/4}$ in the exponent. So when $p$ is small,the bound for $Pr(\textbf{X}-r\textbf{Y} \ge 0)$ is much sharper in (5). The next section provides an inequality which is yet a bit sharper than (4) and applies to any value of $r\ge 1$ and any $0 < p \le 1/2$. 
\section{A Direct Inequality for $Pr\left(\textbf{X}-r\textbf{Y}\ge 0 \right)$}
In this section we develop an inequality that directly addresses $Pr\left(\textbf{X}-r\textbf{Y} \ge 0 \right)$, as opposed to using one of the existing general inequalities. By treating this specific case we will obtain a result that is sightly sharper and more uniform in $p,\epsilon$ than the one in (4) derived from the more general Bennett inequality above. 
We aproach the problem directly using the standard Chernoff inequality and obtain the result below. 
\begin{theorem}
    For $0 < p \le 1/2 \text{ and } \epsilon \ge 0$ we have 
\begin{equation}
     Pr(\textbf{X}-r\textbf{Y}\ge 0) \le \exp \left(-\frac{p}{2q} \frac{K\epsilon^2}{1+r^2}\right) 
\end{equation}   
    where $q=1-p,r=1+\epsilon.$
\end{theorem}
    
Remark: If we consider mean-zero random variables, we have 
 $$Pr(\textbf{X}-r\textbf{Y} \ge 0)=Pr((\textbf{X}-Kp)-r(\textbf{Y}-Kp) \ge Kp\epsilon) $$ 
 $$Var((\textbf{X}-Kp)-r(\textbf{Y}-Kp))=Kpq(1+r^2)=\sigma^2$$ so that the
 exponential in equation (6) has argument $\displaystyle{-\frac{(Kp\epsilon)^2}{2\sigma^2}}$.
\begin{proof}
We use the Chernoff bound -
$$Pr(\textbf{U}\ge a) \le e^{-at}E(\exp(t\textbf{U}))$$
where we identify $\textbf{U}$ as $\textbf{X}-r\textbf{Y}$ and $a=0$. The value of $t$ is chosen to obtain a good bound on the right hand side. 
We have then $$E(\exp(t(\textbf{X}-r\textbf{Y})))=E(\exp(t\textbf{X}))E(\exp(-tr\textbf{Y}))=[(q+pe^t)(q+pe^{-rt})]^K.$$
Since we are looking for exponential bounds we define $F(t)=\log[(q+pe^t)(q+pe^{-rt})]$ so that $E(\exp(t(\textbf{X}-r\textbf{Y})))=\exp(KF(t))$ and then find an upper bound for the minimum over $t$ of $F(t)$.

\noindent We consider the quadratic Taylor polynomial of $F(t)$ about $t=0$, denoted $T_2(t)$:
$$T_2(t)=-(p\epsilon)t+\frac{1}{2} pq(1+r^2)t^2.$$
The minimum of $T_2(t)$ occurs at $\displaystyle{t_{min}=\frac{\epsilon}{q(1+r^2)}}$ and we have $\displaystyle{T_2(t_{min})=-\frac{1}{2} \frac{p}{q} \frac{\epsilon^2}{1+r^2}}$. We will show that $F(t_{min}) \le T_2(t_{min})$ and the result will follow from that.
The function $T_2(t_{min})-F(t_{min})$ is defined on the semi-infinite rectangular domain $(p,r) \in [0,1/2]\times [1,\infty)$. We compactify the domain by defining     $\displaystyle{r=\frac{1}{1-s}}$ for $0 \le s\le 1$, observing that $R(p,s) \equiv T_2(t_{min})-F(t_{min})$ is defined and continuous on the domain $(p,s) \in [0,1/2]\times [0,1]$. We observe that when $p=0$, $q=1$ so $T_2(t) \equiv 0$ and $F(t)=\log(1)=0$ so $R(0,s)=0$ and when $s=0$ we have $r=1,\epsilon=0$ so $t_{min}=0$ and $R(p,0)=0$. We will divide $R(p,s)$ by the leading terms in $p,s$.
We have $\displaystyle{r=\frac{1}{1-s},\epsilon=r-1=\frac{s}{1-s}}$ which leads to
\begin{align}
    \frac{1}{p}[T_2(t_{min})-F(t_{min})]&=-\frac{1}{2(1-p)}\frac{s^2}{(1-s)^2+1} \nonumber\\& 
    -\frac{1}{p}\log\left(1-p+p\exp\left(\frac{s(1-s)}{(1-p)((1-s)^2+1)}\right)\right)\nonumber\\ & 
    -\frac{1}{p}\log\left(1-p+p\exp\left(-\frac{s}{(1-p)((1-s)^2+1)}\right)\right) \nonumber
\end{align}
and so 
\begin{align}
    \lim_{p \to 0}\frac{1}{p}[T_2(t_{min})-F(t_{min})]&=-\frac{1}{2}\frac{s^2}{(1-s)^2+1}\nonumber\\& 
    -\exp\left(\frac{s(1-s)}{((1-s)^2+1)}\right)-\exp\left(-\frac{s}{((1-s)^2+1)}\right)+2. \nonumber
\end{align}
The point here is that this limit exists and (as we will verify later) is strictly positive on the side of the rectangular domain given by $p=0,0 < s\le1$ and that $\displaystyle{\frac{1}{p}[T_2(t_{min})-F(t_{min})]}$ is defined and continuous on the domain. Similarly we develop the behavior near $s=0$:
$$\frac{1}{p} [T_2(t_{min})-F(t_{min})]=s^4\frac{18p^2-30p+11}{192(1-p)^3}+O(s^5) $$
If we now consider the function $\displaystyle{\frac{1}{ps^4} \left[T_2(t_{min})-F(t_{min})\right]} $, this function is differentiable on the compact domain. By factoring out the local behavior where the function is zero on the boundary, we have obtained a function that can be directly verified as strictly positive, via computation (using points spaced by .01 in $p$ and in $s$). A plot of the surface appears in Figure 1. The minimum occurs at $p=.5,s=0$ and is equal to $\displaystyle{\frac{1}{48}}$. The bounding curves are plotted in red. 
\begin{figure}[htp]
    \centering
    \includegraphics[width=15cm]{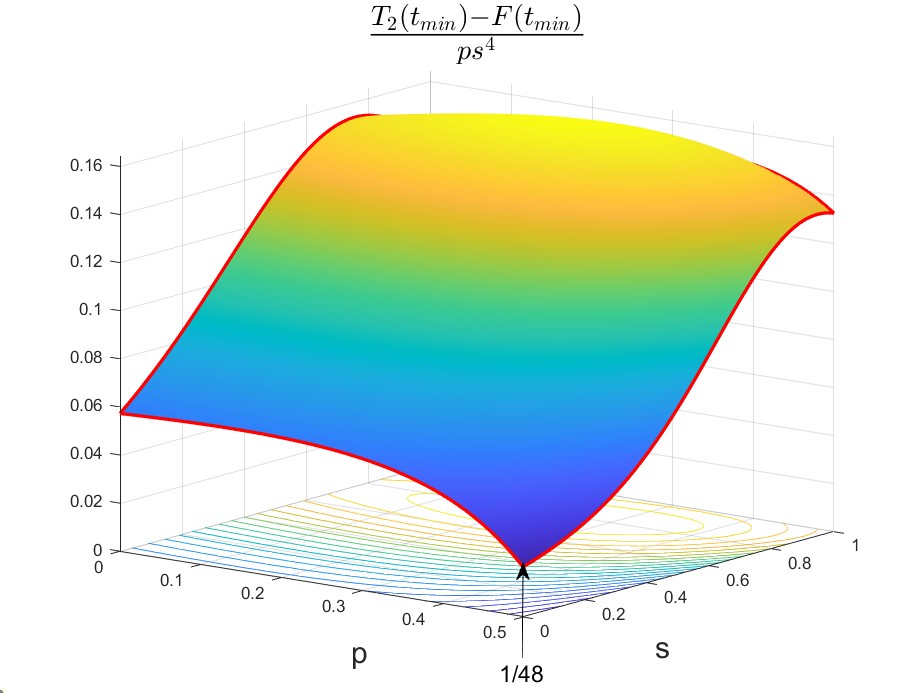}
    \caption{Surface of the scaled difference on domain}
    \label{fig:galaxy}
\end{figure}

This completes the proof. 
\end{proof}
We now can apply the uniform inequality
$$ Pr(\textbf{X}-r\textbf{Y} \ge 0) \le \exp \left(-\frac{p}{2q} \frac{K\epsilon^2}{1+r^2}\right) $$
for $0 \le p \le 1/2$ and any $\epsilon \ge 0$. Using Lemma 2.1 and the fact that $ n!=\exp(n \log n) o(1)$ we can obtain the simple result:
\begin{theorem}
    For any fixed $\epsilon>0$ if $\displaystyle{C\frac{\epsilon^2}{4(1+(1+\epsilon)^2)}>1}$ then $\displaystyle{p=C\frac{\log  n}{n}}$ implies 
    $$\lim_{n \to \infty} Pr(\frac{\textbf X^*}{\textbf Y^*} \ge 1+\epsilon) =0.$$
\end{theorem}

\section {Summary}
In contemplating the Minimal Feedback Arc Set problem, one might imagine that reordering the vertices of a directed graph would provide sufficient power so as to appreciably reduce the size of the feedback arc set relative to the feedforward arcs. Our results show that at least in the case of directed Erd\H{o}s-R\'{e}nyi random graphs, this is not at all the case for large random graphs with a fixed edge probability $p>0$: asymptotically, the ratio of feedforward to feedback arcs approaches 1, even when optimized. Our inequalities show that this remains true for graphs with $p=C\frac{\log n}{n}$: if we scale $C$ appropriately with  $\epsilon$ then the probability of a feedback ratio larger than $1+\epsilon$ similarly approaches zero asymptotically for large graphs.

\section{References}
[Be] G. Bennett. Probability inequalities for the sum of independent random variables. Amer. Stat. Assoc. J., 57:33–45, 1962. \par \noindent
[Bou] S. Boucheron, G. Lugosi, O. Bousquet, \textit{Concentration Inequalities}, in \textbf{Advanced Lectures on Machine Learning} , Springer 2004 \par \noindent
[Ha] Frank Harary, \textbf{Graph Theory}, Addison-Wesley 1969 \par \noindent
[Hef] Hefetz, Steger, Sudakov, \textit{Random Directed Graphs are Robustly Hamiltonian}, Random Struct. Alg., 49,345-362, 2016.  \par \noindent
[Ku] Robert Kudelić, \textbf{Feedback arc set - a history of the problem and algorithms}. Springer, 2022 \par \noindent
[Ho] Wassily Hoeffding. \textit{Probability inequalities for sums of bounded random variables} (PDF). Journal of the American Statistical Association. 58 (301): 13–30. (1963) doi:10.1080/01621459.1963.10500830.
\end{document}